\documentclass[11pt, reqno]{amsart}

\usepackage{amsmath,amssymb,amsthm,amsfonts,verbatim}
\usepackage{microtype}
\usepackage[all,2cell]{xy}
\usepackage{mathtools}
\usepackage{graphicx}
\usepackage{pinlabel}
\usepackage{hyperref}
\usepackage{mathrsfs}
\usepackage{color}
\usepackage[dvipsnames]{xcolor}
\usepackage{enumitem}
\usepackage{cite}
\usepackage{soul}

\CompileMatrices

\usepackage[top=1.2in,bottom=1.2in,left=1in,right=1in]{geometry}

\usepackage{hyperref} 
\hypersetup{          
	colorlinks=true, breaklinks, linkcolor=[RGB]{51 102 204}, filecolor=Orchid, urlcolor=[RGB]{51 102 204},
	citecolor=Orchid, linktoc=all, }
\usepackage[nameinlink]{cleveref}

\theoremstyle{plain}
\newtheorem{theorem}{Theorem}[section]

    \AddToHook{env/maincor/begin}{\crefalias{maintheorem}{maincor}}
    \crefname{maincor}{corollary}{corollaries}

    \AddToHook{env/problem/begin}{\crefalias{theorem}{problem}}
    \crefname{problem}{problem}{problems}

    \AddToHook{env/proposition/begin}{\crefalias{theorem}{proposition}}
    \crefname{proposition}{proposition}{propositions}
\newtheorem{lemma}[theorem]{Lemma}
    \AddToHook{env/lemma/begin}{\crefalias{theorem}{lemma}}
    \crefname{lemma}{lemma}{lemmas}
\newtheorem{conjecture}[theorem]{Conjecture}
    \AddToHook{env/conjecture/begin}{\crefalias{theorem}{conjecture}}
    \crefname{conjecture}{conjecture}{conjectures}

    \AddToHook{env/fact/begin}{\crefalias{theorem}{fact}}
    \crefname{fact}{fact}{facts}

    \AddToHook{env/corollary/begin}{\crefalias{theorem}{corollary}}
    \crefname{corollary}{corollary}{corollaries}

    \AddToHook{env/claim/begin}{\crefalias{theorem}{claim}}
    \crefname{claim}{claim}{claims}

\theoremstyle{definition}
\newtheorem{definition}[theorem]{Definition}
    \AddToHook{env/definition/begin}{\crefalias{theorem}{definition}}
    \crefname{definition}{definition}{definitions}

    \AddToHook{env/example/begin}{\crefalias{theorem}{example}}
    \crefname{example}{example}{examples}

    \AddToHook{env/construction/begin}{\crefalias{theorem}{construction}}
    \crefname{construction}{construction}{constructions}

    \AddToHook{env/exercise/begin}{\crefalias{theorem}{exercise}}
    \crefname{exercise}{exercise}{exercises}

    \AddToHook{env/remark/begin}{\crefalias{theorem}{remark}}
    \crefname{remark}{remark}{remarks}

    \AddToHook{env/question/begin}{\crefalias{theorem}{question}}
    \crefname{question}{question}{questions}

    \AddToHook{env/answer/begin}{\crefalias{theorem}{answer}}
    \crefname{answer}{answer}{answers}

    \AddToHook{env/convention/begin}{\crefalias{theorem}{convention}}
    \crefname{convention}{convention}{conventions}

\newcommand{\nc}{\newcommand}
\nc{\dmo}{\DeclareMathOperator}
\nc{\OO}{\mathcal{O}}
\nc{\cA}{\mathcal{A}}
\nc{\cB}{\mathcal{B}}
\nc{\sB}{\mathscr{B}}
\nc{\C}{\mathbb{C}}
\nc{\cC}{\mathcal{C}}
\nc{\sC}{\mathscr{C}}
\nc{\BB}{\mathbb{B}}
\nc{\LL}{\mathcal{L}}
\nc{\bd}{\mathbf{d}}
\nc{\DD}{\mathbb{D}}
\nc{\cD}{\mathcal{D}}
\nc{\sD}{\mathscr{D}}
\nc{\bF}{\mathbb{F}}
\nc{\cF}{\mathcal{F}}
\nc{\cG}{\mathcal{G}}
\nc{\cI}{\mathcal{I}}
\nc{\cK}{\mathcal{K}}
\nc{\cL}{\mathcal{L}}
\nc{\cM}{\mathcal{M}}
\nc{\bM}{\mathbf{M}}
\nc{\N}{\mathbb{N}}
\nc{\cN}{\mathcal{N}}
\nc{\cO}{\mathcal{O}}
\nc{\bp}{\mathbf{p}}
\nc{\bP}{\mathbb{P}}
\nc{\cP}{\mathcal{P}}
\nc{\Q}{\mathbb{Q}}
\nc{\R}{\mathbb{R}}
\nc{\cS}{\mathcal{S}}
\nc{\cT}{\mathcal T}
\nc{\cU}{\mathcal U}
\nc{\bV}{\mathbb V}
\nc{\cX}{\mathcal{X}}
\nc{\cY}{\mathcal{Y}}
\nc{\Z}{\mathbb{Z}}
\nc{\disk}{\mathbb{D}}
\nc{\hyp}{\mathbb{H}}

\nc{\CP}{\mathbb{CP}}
\nc{\RP}{\mathbb{RP}}
\dmo{\Mod}{Mod}
\dmo{\br}{B}
\dmo{\PMod}{PMod}
\dmo{\LMod}{LMod}
\dmo{\Diff}{Diff}
\dmo{\Homeo}{Homeo}
\dmo{\dist}{dist}
\dmo\BDiff{BDiff}
\dmo\SO{SO}
\dmo\Sp{Sp}
\dmo\Hom{Hom}
\dmo\SL{SL}
\dmo\rank{rank}
\dmo\sig{sig}
\dmo\Out{Out}
\dmo\Aut{Aut}
\dmo\Inn{Inn}
\dmo\GL{GL}
\dmo\PGL{PGL}
\dmo\Gr{Gr}
\dmo\PSL{PSL}
\dmo\BHomeo{BHomeo}
\dmo\EHomeo{EHomeo}
\dmo\EDiff{EDiff}
\dmo\Disc{Disc}
\dmo\Aff{Aff}

\renewcommand{\hat}{\widehat}
\renewcommand{\bar}{\overline}

\dmo\Teich{Teich}
\dmo\Fix{Fix}
\nc{\pair}[1]{\ensuremath{\left\langle #1 \right\rangle}}
\nc{\abs}[1]{\ensuremath{\left| #1 \right|}}
\nc{\action}{\circlearrowright}
\nc{\abcd}[4]{\ensuremath{\left(\begin{array}{cc} #1 & #2 \\ #3 & #4 \end{array}\right)}}
\nc{\into}\hookrightarrow
\dmo{\Isom}{Isom}
\nc{\normal}{\vartriangleleft}
\dmo{\Vol}{Vol}
\dmo{\im}{Im}
\dmo{\Push}{Push}
\dmo{\Conf}{Conf}
\dmo{\PConf}{PConf}
\dmo{\PB}{PB}
\dmo{\id}{id}
\dmo{\Jac}{Jac}
\dmo{\Pic}{Pic}
\dmo{\Stab}{Stab}
\dmo{\Arf}{Arf}
\dmo{\End}{End}
\dmo{\Ext}{Ext}
\dmo{\Gal}{Gal}
\dmo{\lcm}{lcm}
\dmo{\ab}{ab}
\dmo{\opp}{op}
\dmo{\SU}{SU}
\dmo{\OT}{\Omega \mathcal{T}}
\dmo{\OM}{\Omega \mathcal{M}}
\dmo{\PH}{\mathbb{P}\mathcal{H}}
\dmo{\spin}{spin}
\dmo{\even}{even}
\dmo{\odd}{odd}
\dmo{\comp}{\mathcal{H}}
\dmo{\Mgk}{\mathcal{M}_{g, \underline{\kappa}}}
\dmo{\orb}{orb}
\dmo{\AJ}{AJ}
\dmo{\Ck}{\mathsf{C}(\underline{\kappa})}
\dmo{\Int}{Int}
\dmo{\pr}{pr}
\dmo{\lab}{lab}
\dmo{\Sym}{Sym}
\dmo{\Ann}{Ann}
\dmo{\Rad}{Rad}
\dmo{\Ind}{Ind}
\dmo{\Div}{Div}
\dmo{\Res}{Res}
\dmo{\Hur}{Hur}
\dmo{\vcd}{vcd}
\dmo{\codim}{codim}

\nc{\Span}[1]{\operatorname{Span}(#1)}

\renewcommand{\epsilon}{\varepsilon}
\renewcommand{\tilde}{\widetilde}

\renewcommand{\ge}{\geqslant}
\nc{\coloneq}{\mathrel{\mathop:}\mkern-1.2mu=}
\nc{\margin}[1]{\marginpar{\scriptsize #1}}
\nc{\para}[1]{\medskip\noindent\textbf{#1.}}
\definecolor{myblue}{RGB}{102,153, 255}
\definecolor{myred}{RGB}{204,0,0}
\definecolor{mygreen}{RGB}{0,204,0}
\definecolor{myorange}{RGB}{255,102,0}
\definecolor{mypurple}{RGB}{138,43,226}
\nc{\red}[1]{\textcolor{myred}{#1}}
\nc{\blue}[1]{\textcolor{myblue}{#1}}

\usepackage{tikz}
\usepackage{amsmath,amssymb}
\usepackage{tikz-cd}
\usetikzlibrary{decorations.markings,calc}

\author[Shile Wang]{Shile Wang}

\address{School of Mathematical Sciences, Peking University, Beijing, China}
\email{wangshile@stu.pku.edu.cn}

\title{A Faithful Action of the Mapping Class Group $\mathrm{Mod}(S_g)$ on the $(2g-1)$-Sphere}

\date{\today}

\begin{document}

\begin{abstract}
We prove that $\mathrm{Mod}(S_g)$ acts faithfully by orientation-preserving homeomorphisms on a $(2g-1)$-sphere when $g \geq 3$ using equivariant blow-up and bi-orderability of the Torelli group.
\end{abstract}
\maketitle{}

\section{Introduction}

Let $S_g$ be the connected, closed, orientable surface of genus $g$, and $\mathbb{S}^n$ be the $n$-sphere. Let $\mathrm{Homeo}^+(X)$ be the orientation-preserving homeomorphism group of an orientable manifold $X$. The mapping class group $\mathrm{Mod}(X)$ is the group of isotopy classes of orientation-preserving homeomorphisms of $X$, that is, $\pi_0(\mathrm{Homeo}^+(X))$.

In this paper, we focus on the injective homomorphisms from $\mathrm{Mod}(S_g)$ to $\mathrm{Homeo}^+(\mathbb{S}^{n})$. Much work has been done on the homomorphisms between mapping class groups, braid groups, and transformation groups of particular manifolds, concerning the constraints on genus and dimensions; see work by Chen, De Pool, Franks, Handel in \cite{pure_braid_group}, \cite{Braid_group}, \cite{pure_MCG}, \cite{Triviality} for more details. In particular, Chen and Lanier \cite{CL} considered homomorphisms from $\mathrm{Mod}(S_g)$ to $\mathrm{Homeo}^+(\mathbb{S}^{n})$, and proved several results on the triviality of homomorphisms when $n$ is small by analyzing finite subgroups of $\mathrm{Mod}(S_g)$.

Although the corresponding result on higher-dimensional spheres remains unknown, there exist nontrivial homomorphisms from $\mathrm{Mod}(S_g)$ to $\mathrm{Homeo}^+(\mathbb{S}^{n})$ when $n$ is large enough. A simple example is the action ginve by the symplectic representation $\mathrm{Mod}(S_g) \to \mathrm{Sp}(2g,\mathbb{Z})$ and the action of $\mathrm{Sp}(2g,\mathbb{Z})$ on the oriented projective space of $\mathbb{R}^{2g}$. Zimmermann \cite{Symplectic_groups} proved that the minial $n$ such that there exists a nontrivial homomorhism $\mathrm{Sp}(2g,\mathbb{Z}) \to \mathrm{Homeo}^+(\mathbb{S}^{n})$ is $n=2g-1$. If we further assume that the homomorphism is injective, there is also a natural injective homomorphism from $\mathrm{Mod}(S_g)$ to $\mathrm{Homeo}^+(\mathbb{S}^{6g-7})$ given by the action of $\mathrm{Mod}(S_g)$ in $\mathcal{PMF}(S_g) \approx \mathbb{S}^{6g-7}$, the space of projective measured foliations in $S_g$.

In this paper, we give an injective homomorphism $\mathrm{Mod}(S_g) \to \mathrm{Homeo}^+(\mathbb{S}^{2g-1})$ by constructing a faithful action of $\mathrm{Mod}(S_g)$ on $\mathbb{S}^{2g-1}$ when $g \geq 3$. Our main result is the following.

\begin{theorem}\label{main}
For $g \geq 3$, there exists an injective homomorphism
\[
\mathrm{Mod}(S_g) \to \mathrm{Homeo}^+(\mathbb{S}^{2g-1}).
\]
\end{theorem}

Chen and Lanier \cite[Conjecture~9(2)]{CL} proposed their conjecture that the minimal dimension $n$ of a sphere such that there exists an injective homomorphism $\mathrm{Mod}(S_g) \to \mathrm{Homeo}^+(\mathbb{S}^{n})$ is $6g-7$ when $g \geq 3$. Thus, our Theorem \ref{main} disproves their conjecture and we have the following conjecture:

\begin{conjecture}
For $g \geq 3$:
\begin{enumerate}
    \item The minimal $n$ such that there exists a nontrivial homomorphism
\[
\mathrm{Mod}(S_g) \to \mathrm{Homeo}^+(\mathbb{S}^{n})
\]
is $n=2g-1$.
    \item The minimal $n$ such that there exists an injective homomorphism
\[
\mathrm{Mod}(S_g) \to \mathrm{Homeo}^+(\mathbb{S}^{n})
\]
is $n=2g-1$.
\end{enumerate}
\end{conjecture}

Our strategy to construct the counterexample is as follows. We will use the symplectic representation of $\mathrm{Mod}(S_g)$
\[
1 \longrightarrow \mathcal{I}(S_g) \longrightarrow \mathrm{Mod}(S_g) \stackrel{\Psi}{\longrightarrow} \mathrm{Sp}(2g,\mathbb{Z}) \longrightarrow 1
\]
to extend the natural action of $\mathrm{Sp}(2g,\mathbb{Z})$ on $\mathbb{P}^+(\mathbb{R}^{2g}) \approx \mathbb{S}^{2g-1}$ to a faithful action of $\mathrm{Mod}(S_g)$ on $\mathbb{S}^{2g-1}$ when $g \geq 3$. The extension is realized in three steps. First, we apply the technique of equivariant blow-up to blow up a dense $\mathrm{Sp}(2g,\mathbb{Z})$-orbit in $\mathbb{S}^{2g-1}$ and lift the original action of $\mathrm{Sp}(2g,\mathbb{Z})$ to an equivariant action on the space obtained by blowing up; see Theorem \ref{Blow_up}. Second, we construct a faithful action of $\mathcal{I}(S_g)$ on a closed $(2g-1)$-ball using the fact that $\mathcal{I}(S_g)$ is a left-orderable group. Finally, we extend the action of $\mathrm{Sp}(2g,\mathbb{Z})$ constructed in the first step to an action of $\mathrm{Mod}(S_g)$ by making use of the structure of the equivariant blow-up and the faithful action on closed $(2g-1)$-balls constructed in the second step.

The main difficulty in producing this construction is to notice that $\mathcal{I}(S_g)$ is a left-orderable group. Although this fact is not hard to prove using the classic results established in the study of mapping class groups, it has not received much attention before. We give a proof of this fact based on the classic results on Johnson filtration. Once we observe this fact, it is natural to apply the characterization of left-orderable groups to produce the faithful action on closed $n$-balls. On the other hand, the technique of equivariant blow-up we used in the first step of construction is standard in geometry and topology. We prove it based on the original proof given by Bowden and Mann \cite{equivariant_blow-up}. With the establishment of the first two steps, the third step of construction is straightforward. In this paper, we state and prove all the results in the general case and then apply them to the case of $\mathrm{Mod}(S_g)$.

This paper is organized as follows. In Section \ref{mapping_class_group_preliminaries}, we recall some classic results in the study of mapping class groups. In particular, we prove that the Torelli group $\mathcal{I}(S_g)$ is left-orderable when $g \geq 3$; see Theorem \ref{torelli_bi_orderable}. In Section \ref{equivariant_blow_up}, we introduce the technique of equivariant blow-up. More precisely, for a group action by diffeomorphisms on a $n$-sphere and a dense orbit of it, we blow up each single point $x$ in the dense orbit by replacing it with a closed $n$-ball $D^n$, where $\partial D^n$ is identified with the oriented projective space of the tangent space at $x$. We show that the space obtained by blowing up is still a $n$-sphere, and the original action lifts to an equivariant action by homeomorphisms on it. See Theorem \ref{Blow_up} for more details. In Section \ref{faithful_action_on_closed_ball}, we construct a faithful action on $D^m$ for every left-orderable group; see Theorem \ref{act_faithfully_on_closed_disks}. In Section \ref{proof_of_main_thm}, we prove that given a short exact sequence
\[
1 \longrightarrow K \longrightarrow G \stackrel{\pi}{\longrightarrow} Q \longrightarrow 1
\]
and a faithful action of $K$ on a closed ball, one can lift an equivariant action of $Q$ as described in Theorem \ref{Blow_up} to an action of $G$. In the end of Section \ref{proof_of_main_thm}, we prove Theorem \ref{main} by applying this theorem and using the results established above.

\section{Mapping Class Group Preliminaries}\label{mapping_class_group_preliminaries}
We first recall some classic results on mapping class groups. We focus on the finite generation, the symplectic representation, and the Johnson filtration of $\mathrm{Mod}(S_g)$. In the end of this section, we prove that the Torelli group $\mathcal{I}(S_g)$ is bi-orderable for $g \geq 3$.

\subsection{Finite generation}

The mapping class group $\mathrm{Mod}(S_g)$ is finitely generated. More precisely, we have the following:

\begin{theorem}[Dehn–Lickorish theorem, Theorem 4.1, \cite{Primer}]
For $g \geq 0$, $\mathrm{Mod}(S_g)$ is generated by finitely many Dehn twists about nonseparating simple closed curves.
\end{theorem}

In particular, when $g \geq 2$, $\mathrm{Mod}(S_g)$ can be generated by $3g-1$ elements called the \textit{Lickorish generators}, or by $2g+1$ elements called the \textit{Humphries generators}. See Theorem 4.13, Theorem 4.14 in \cite{Primer} for a detailed proof.

\subsection{Symplectic representation}

A important way to study the property of $\mathrm{Mod}(S_g)$ is by considering its action on $H_1(S_g;\mathbb{Z}) \cong \mathbb{Z}^{2g}$, because we can understand $\mathrm{Mod}(S_g)$ through the well-developed theory on the linear group $\mathrm{Aut}(H_1(S_g;\mathbb{Z}))$.

Let
\[
\hat{i} \colon H_1(S_g;\mathbb{Z}) \wedge H_1(S_g;\mathbb{Z}) \to \mathbb{Z}
\]
be the algebraic intersection number. It extends uniquely to a nondegenerate, alternating bilinear map
\[
\hat{i} \colon H_1(S_g;\mathbb{R}) \wedge H_1(S_g;\mathbb{R}) \to \mathbb{R}.
\]
Then $(H_1(S_g;\mathbb{R}),\hat{i})$ is a symplectic vector space. Since $\mathrm{Mod}(S_g)$ preserves $\hat{i}$ when acting on $H_1(S_g;\mathbb{Z})$, we have a representation $\Psi \colon \mathrm{Mod}(S_g) \to \mathrm{Sp}(2g,\mathbb{Z})$ called the \textit{symplectic representation} of $\mathrm{Mod}(S_g)$.

One of the most important properties of the symplectic representation is its surjectivity. See \cite{Primer} for three different proofs of the following theorem.

\begin{theorem}[Theorem 6.4, \cite{Primer}]
The representation $\Psi:\mathrm{Mod}(S_g) \to \mathrm{Sp}(2g,\mathbb{Z})$ is surjective when $g \geq 1$.
\end{theorem}

The kernel of $\Psi$ is called the \textit{Torelli group}, denoted $\mathcal{I}(S_g)$. Then we have a short exact sequence
\[
1 \longrightarrow \mathcal{I}(S_g) \longrightarrow \mathrm{Mod}(S_g) \stackrel{\Psi}{\longrightarrow} \mathrm{Sp}(2g,\mathbb{Z}) \longrightarrow 1.
\]

\subsection{Johnson filtration}
Let $g \geq 2$ and $\Gamma = \pi_1(S_g)$. We take the lower central series of $\Gamma$
\[
\Gamma = \Gamma_1 \supset \Gamma_2 \supset \cdots
\]
defined inductively by
\[
\Gamma_1 = \Gamma \quad \text{and} \quad \Gamma_i = [\Gamma,\Gamma_{i-1}].
\]
There exists a natural homomorphism
\[
\Psi_i \colon \mathrm{Mod}(S_g) \to \mathrm{Out}(\Gamma/\Gamma_{i+1}).
\]
Define the $i^{\text{th}}$ Torelli group as $\mathcal{I}^i(S_g) = \mathrm{ker}(\Psi_i)$. This gives a filtration $\{\mathcal{I}^i(S_g)\}_{i=1}^\infty$ of $\mathcal{I}(S_g)$, called the \textit{Johnson filtration}. The Johnson filtration is also denoted $\{J_i\}_{i=1}^\infty$. Magnus \cite{Lower_central_series} proved that the intersection $\bigcap_{i=1}^{\infty} \Gamma_i$ is trivial. Then Bass and Lubotzky \cite{Johnson_filtration} proved that the intersection $\bigcap_{i=1}^{\infty} \mathcal{I}^i(S_g)$ is also trivial. More precisely, we have the following:

\begin{theorem}[Bass and Lubotzky, \cite{Johnson_filtration}]\label{torsion-free}
Let $S$ be a closed orientable surface with genus $g$ and $n$ marked points. Then the Johnson filtration $\{J_i\}_{i=1}^\infty$ on $\mathcal{I}(S)$ satisfies:
\begin{enumerate}
    \item $\bigcap_{i=1}^{\infty} J_i \subset Z(\mathrm{Mod}(S))$.
    \item $J_i/J_{i+1}$ is torsion-free for every $i \geq 1$.
\end{enumerate}
\end{theorem}

\subsection{Bi-orderablility of the Torelli group}

The properties of left-orderability and bi-orderability show up in this subsection because they are naturally related to dynamics and produce the faithful actions of $\mathcal{I}(S_g)$ we need; see Theorem \ref{act_faithfully_on_closed_disks}. We first recall the definition of left-orderable groups and bi-orderable groups.

\begin{definition}
A group $G$ is \textit{left-orderable} if there exists a total order $<$ on $G$ that is invariant under left multiplication, i.e. for every $x,y,g \in G$, if $x<y$ then $gx<gy$. Similarly, we can define \textit{right-orderable} groups. 
\end{definition}

\begin{definition}
A group $G$ is called \textit{bi-orderable} if there exists a total order $<$ on $G$ that is invariant under both left and right multiplication, i.e. for every $x,y,g \in G$, if $x<y$ then $gx<gy$ and $xg<yg$.
\end{definition}

One may refer to Fuchs's book \cite{Partially_ordered_algebraic_systems} to see more examples and properties of left-orderable and bi-orderable groups. In particular, we point out the following:

\begin{lemma}\label{torsion-free_abelian_bi-orderable}
A torsion-free abelian group $G$ is bi-orderable.
\end{lemma}

\begin{proof}
Since $G$ is torsion-free, we have a natural injective homomorphism $G \to V = G \otimes_{\mathbb{Z}} \mathbb{Q}$. By the Axiom of Choice, there exists an ordered basis of $V$. This gives a lexicographic order on $V$, which restricts to $G$. One can check directly from definition that this order is invariant under multiplication.
\end{proof}

\begin{theorem}\label{torelli_bi_orderable}
Let $g \geq 3$. Then $\mathcal{I}(S_g)$ is bi-orderable. In particular, it is left-orderable.
\end{theorem}

Our strategy to prove this theorem is to construct an order on $\mathcal{I}(S_g)$ and check that it is invariant under multiplication. Note that Theorem \ref{torsion-free} show that $J_i/J_{i+1}$ is torsion-free. Thus, we can apply Lemma \ref{torsion-free_abelian_bi-orderable} to $J_i/J_{i+1}$. Then we make use of the Johnson filtration to construct a total order on $\mathcal{I}(S_g)$ based on the order on $J_i/J_{i+1}$.

Recall that we have $[\Gamma_i,\Gamma_j] \subseteq \Gamma_{i+j}$ for every $i,j \geq 1$. Before we prove Theorem \ref{torelli_bi_orderable}, we first prove that the Johnson filtration also has the similar property in the following lemma.

\begin{lemma}
Let $g \geq 3$, and $\{J_i\}_{i=1}^\infty$ be the Johnson filtration of $\mathcal{I}(S_g)$. Then $[J_i,J_j] \subseteq J_{i+j}$ for every $i,j \geq 1$.
\end{lemma}

\begin{proof}
Let $\Gamma_i$ be defined as above. Let $\phi \in J_i$ and $\Phi$ be an automorphism of $\Gamma$ representing $\phi$. Since $\phi$ induces an inner automorphism of $\Gamma/\Gamma_{i+1}$, we can modify $\Phi$ by composing it with some inner automorphism such that $\Phi(x)x^{-1} \in \Gamma_{i+1}$ for every $x \in \Gamma$.

We now prove by induction on $j \geq 1$ that for every $x \in \Gamma_j$, $\Phi(x)x^{-1} \in \Gamma_{i+j}$. The case $j=1$ was already proved in the paragraph above. Assume that the assertion is correct for $1 \leq j \leq k$. For every $z \in \Gamma_{k+1}$, we can write $z = [x,y]$, where $x \in \Gamma$, $y \in \Gamma_k$. By the induction hypothesis,
\[
a = \Phi(x)x^{-1} \in \Gamma_{i+1}, \quad b = \Phi(y)y^{-1} \in \Gamma_{i+k}.
\]
Then
\[
\Phi(z)z^{-1} = [\Phi(x),\Phi(y)]z^{-1} = axbyx^{-1}a^{-1}y^{-1}b^{-1}z^{-1}.
\]
Note that $[\Gamma_r,\Gamma_s] \subseteq \Gamma_{r+s}$ for every $r,s \geq 1$. In particular,
\[
[\Gamma_1,\Gamma_{i+k}] \subseteq \Gamma_{i+k+1}, \quad [\Gamma_{i+1},\Gamma_k] \subseteq \Gamma_{i+k+1}.
\]
Then
\[
[b,yx^{-1}a^{-1}y^{-1}] \in \Gamma_{i+k+1}, \quad [a^{-1},y^{-1}] \in \Gamma_{i+k+1}, \quad [a,z] \in \Gamma_{i+k+2} \subset \Gamma_{i+k+1}
\]
and
\[
\Phi(z)z^{-1} \equiv axyx^{-1}a^{-1}y^{-1}z^{-1} \equiv axyx^{-1}y^{-1}a^{-1}z^{-1} \equiv aza^{-1}z^{-1} \equiv [a,z] \equiv 1 \pmod{\Gamma_{i+k+1}}.
\]
Hence $\Phi(z)z^{-1} \in \Gamma_{i+k+1}$, and we have proved the assertion.

Let $\phi \in J_i$, $\psi \in J_j$, and $\Phi$, $\Psi$ be automorphisms that represent $\phi$, $\psi$ respectively and satisfy the above property. For $x\in \Gamma$, we write
\[
a = \Phi(x)x^{-1} \in \Gamma_{i+1}, \quad b = \Psi(x)x^{-1} \in \Gamma_{j+1}.
\]
Note that
\[
\Phi(b)b^{-1} \in \Gamma_{i+j+1}, \quad \Psi(a)a^{-1} \in \Gamma_{i+j+1}, \quad [a,b] \in \Gamma_{i+j+2} \subset \Gamma_{i+j+1}.
\]
Then
\[
\Phi \Psi(x) \equiv \Phi(bx) \equiv bax \equiv abx \equiv \Psi(ax) \equiv \Psi \Phi(x) \pmod{\Gamma_{i+j+1}}.
\]
Hence $[\Phi,\Psi](x)x^{-1} \in \Gamma_{i+j+1}$ for every $x \in \Gamma$. By definition, $[\phi,\psi] \in J_{i+j}$.
\end{proof}

Now we are ready to prove Theorem \ref{torelli_bi_orderable}.

\begin{proof}[Proof of Theorem \ref{torelli_bi_orderable}]
By Theorem \ref{torsion-free}, $J_i/J_{i+1}$ is torsion-free for every $i \geq 1$. Then Lemma \ref{torsion-free_abelian_bi-orderable} gives a total order $<_k$ on $J_k/J_{k+1}$ that is invariant under multiplication. Since $\bigcap_{i=1}^{\infty} J_i \subset Z(\mathrm{Mod}(S))$ is trivial, there exists a largest $l(x) \in \mathbb{Z}_+$ such that $x \in J_{l(x)}$ for every nontrivial $x \in \mathcal{I}(S_g)$. Define a total order $<$ on $\mathcal{I}(S_g)$ as the following. Let $x$ be positive if and only if its nontrivial image in $J_{l(x)}/J_{l(x)+1}$ is positive with the order $<_{l(x)}$. Let $x<y$ if and only if $x^{-1}y$ is positive. Then $<$ is a total order on $\mathcal{I}(S_g)$ invariant under left multiplication. Since $[J_1,J_k] \subseteq J_{k+1}$ for every $k \geq 1$, $gxg^{-1}$ and $x$ have the same image in $J_{l(x)}/J_{l(x)+1}$ for every $g \in \mathcal{I}(S_g)$. Hence $<$ is also invariant under right multiplication, and $\mathcal{I}(S_g)$ is bi-orderable with $<$.
\end{proof}

\section{Equivariant Blow-up of a dense orbit}\label{equivariant_blow_up}

In this section, we introduce the technique of equivariant blow-up in a form that we will apply in our main construction. Generally speaking, given a group action $\rho$ of $G$ on a space $M$, an \textit{equivariant blow-up} of a $G$-invariant space $Z$ consists of:
\begin{enumerate}
    \item a blown-up space $M'$
    \item a lifted group action $\rho'$ of $G$ on $M'$
    \item a surjective equivariant map $h \colon M' \to M$ such that $h^{-1}(z)$ is a single point for $z \in M \setminus Z$ and a specific fiber for $z \in Z$.
\end{enumerate}
Here, the equivariance of $h$ means that $h\rho'(g) = \rho(g)h$ for every $g \in G$.

In the following theorem, we discuss the equivariant blow-up of a dense orbit on a sphere, where $h^{-1}(z)$ is an embedded closed ball for every $z \in Z$. Recall that a family $\mathcal{F}$ of subsets of some metric space is called a \textit{null family} if for every $\epsilon > 0$, there are only finitely many $A \in \mathcal{F}$ such that $\mathrm{diam}(A) > \epsilon$.

\begin{theorem}\label{Blow_up}
Let $\Gamma$ be a countable group. Let $n \geq 1$, and
\[
\rho \colon \Gamma \to \mathrm{Diff}^+(\mathbb{S}^n)
\]
be a group action. Assume that there exists $x_0 \in \mathbb{S}^n$ with trivial $\Gamma$-stabilizer such that its $\Gamma$-orbit $Z$ is dense in $\mathbb{S}^n$. Then there exists a group action
\[
\rho' \colon \Gamma \to \mathrm{Homeo}^+(\mathbb{S}^n)
\]
and a surjective, continuous map $h:\mathbb{S}^n \to \mathbb{S}^n$ that satisfies:
\begin{enumerate}[label=(\arabic*)]
  \item \textbf{Equivariance}: $h\rho'(\gamma) = \rho(\gamma)h$ for every $\gamma \in \Gamma$;
  \item $h^{-1}(\rho(\gamma)x_0)$ is a topologically embedded closed $n$-ball for every $\gamma \in \Gamma$;
  \item $h^{-1}(x)$ is a single point for every $x \in S^n \setminus Z$,
  \item $\{h^{-1}(\rho(\gamma)x_0) : \gamma \in \Gamma \}$ is a null family for some compatible metric on the domain sphere.
\end{enumerate}
\end{theorem}

This theorem comes from Bowden and Mann \cite[Proposition 4.1]{equivariant_blow-up}, and we give a proof in the form used in our main construction.

Although blowing up a finite orbit is a standard construction, we need to be more careful when blowing up a dense orbit to preserve the continuity of the map. The strategy of our proof is to first construct the equivariant blow-up inductively and then take the inverse limit. More precisely, we blow up one point each time by replacing this point with the oriented projective space of its tangent space. To distinguish the topological type of the inverse limit, we make use of the characterization of Sierpiński curves given by Cannon \cite{Sierpiński}.

\begin{definition}[Cannon, \cite{Sierpiński}]\label{Sierpiński_curve}
Let $X$ be a non-empty compact connected metric space. The space $X$ is called an \textit{$(n-1)$-dimensional Sierpiński curve} if there is an embedding $h \colon X \to \mathbb{S}^n$ such that the connected components $U_1, U_2, \cdots$ of $\mathbb{S}^n \setminus h(X)$ form a null family satisfying:
\begin{enumerate}
    \item $\mathbb{S}^n \setminus U_i$ is a $n$-cell for every $i \ge 1$)
    \item $\overline{U_i} \cap \overline{U_j} = \varnothing$ for every $i \neq j$
    \item $\bigcup_{i=1}^{\infty} U_i$ is dense in $\mathbb{S}^n$.
\end{enumerate}
\end{definition}

Cannon proved that there is precisely one $(n-1)$-dimensional Sierpiński curve up to homeomorphism \cite{Sierpiński} when $n \neq 4$. Since Quinn \cite{Quinn} proved the Annulus Theorem in dimension $4$, Cannon's proof can be extended to $n \geq 3$. This result can be stated in different ways; see \cite{Sierpiński} for more details. In particular, Cannon's result implies that every $(n-1)$-dimensional Sierpiński curve is homeomorphic to the standard one whose complement in $\mathbb{S}^n$ is a null family of closed $n$-balls. Then we may conclude that collapsing each boundary component of a $(n-1)$-dimensional Sierpiński curve to a single point gives a $n$-sphere.

\begin{proof}[Proof of Theorem \ref{Blow_up}]
We write $\Gamma=\{\gamma_1, \gamma_2, \cdots\}$, and $z_i = \gamma_i \cdot x_0$. We first construct an equivariant blow-up for a single point.

Equip $X_0=\mathbb{S}^n$ with the standard round metric $d$. Choose a sufficiently small $\epsilon_1>0$. Let $D_1'$ and $D_1$ be the $\epsilon_1$-ball and the $\frac{\epsilon_1}{2}$-ball about $z_1$, respectively. Let $X_1=X_0 \setminus D_1$. Let $f_1 \colon X_1 \to X_0$ be the identity on $X_1 \setminus D_1'$ and be the radial collapse along geodesics through $z_1$ on the annulus $D_1' \setminus D_1$. Then $f_1$ is a smooth surjective map. To ensure that the final construction is a Sierpiński curve, we further require that $d(f_1(x),f_1(y)) < 3 d(x,y)$ for every $x,y \in X_1$. Equivalently, $d(f_1^{-1}(x),f_1^{-1}(y))>\frac{1}{3} d(x,y)$ for every $x,y \in X_0$ and every choice of the preimage $f_1^{-1}(z_1)$. This can be realized due to the definition of $f_1$. Note that $f(\partial D_1) = z_1$ and $f$ is injective outside $\partial D_1$. Hence $f_1^{-1}$ identifies every $v \in \partial D_1$ with an oriented tangent direction to $X_0$ at $z_1$, and every diffeomorphism $\phi$ of $X_0$ that fixes $z_1$ lifts to a homeomorphism $\tilde{\phi}$ of $X_1$ such that $f_1 \tilde{\phi} = \phi f_1$.

We now construct $X_m \subset \mathbb{S}^n$ and a smooth surjective map $f_m \colon X_m \to X_{m-1}$ inductively for every $m \geq 1$. Assume that $X_m$ and $f_m$ are already defined for $1 \leq m \leq k$. Now we construct $X_{k+1}$ and $f_{k+1}$ as follows. Let
\[
F_m = f_1 f_2 \cdots f_m \colon X_m \to X_0.
\]
Choose
\[
\epsilon_{k+1} \leq \max\{\frac{\epsilon_k}{2}, \frac{d(F_k^{-1}(z_{k+1}),\partial X_k)}{2}\},
\]
and let $D_{k+1}'$, $D_{k+1}$ be the closed $\epsilon_{k+1}$-ball, $r_{k+1}\epsilon_{k+1}$-ball about $F_k^{-1}(z_{k+1})$ respectively, where $0 < r_{k+1} <1$ is sufficiently small. Let $X_{k+1} = X_k \setminus D_{k+1}$ and $f_{k+1}$ be defined in the same way as above. We also require
\[
d(f_{k+1}^{-1}(x),f_{k+1}^{-1}(y)) > \lambda_{k+1}d(x,y)
\]
for every $x,y \in X_k$ and every choice of the preimage $f_{k+1}^{-1}(z_{k+1})$, where $0< \lambda_{k+1} <1$. This can be realized for any $\lambda_{k+1}$ sufficiently close to $1$ if we choose $r_{k+1}$ sufficiently small. In particular, we choose $\lambda_{k+1}$ such that $\lambda_1 \lambda_2 \cdots \lambda_{k+1} \to \delta$ for some $\delta >0$. Similarly, $F_{k+1}^{-1}$ identifies every $v \in \partial D_{k+1}$ with an oriented tangent direction to $X_0$ at $z_{k+1}$. Hence, every diffeomorphism $\phi$ of $X_0$ that preserves the set $\{z_1, z_2, \cdots, z_{k+1}\}$ lifts to a homeomorphism $\tilde{\phi}$ of $X_{k+1}$ such that $F_{k+1} \tilde{\phi} = \phi F_{k+1}$. This can be done by conjugation of $F_{k+1}$ to reduce to the case where $\phi$ fixes $z_1, z_2, \cdots, z_{k+1}$ pointwise and the identifications described above.

By the construction above, we conclude that $\{f_k\}$ and $\{F_k\}$ have the following properties:
\begin{enumerate}
    \item For every $k \geq 1$, we have $\sup\limits_{x \in X_k} d(f_k(x), x) \leq \epsilon_k \leq \frac{\epsilon_1}{2^{k-1}} < \frac{1}{2^{k-1}}$.
    \item For every $x,y \in X_0$, we have $d(F_k^{-1}(x), F_k^{-1}(y)) > \delta d(x,y)$ for all $k \geq 1$.
\end{enumerate}

Now, let $X = \varprojlim X_k$ be the inverse limit of the system of maps $X_k \xrightarrow{f_k} X_{k-1}$. In other words,
\[
X = \{(\cdots, p_2, p_1, p_0) \mid p_i \in X_i \text{ and } p_{k-1} = f_k(p_k)\}.
\]
Since $f_k$ is continuous, $X$ is closed and hence is compact. Also, property (1) gives that $\lim_{i \to \infty}p_i$ is well-defined. Then we can define $\Phi \colon X \to \bigcap_kX_k$ by
\[
\Phi(\cdots,p_2, p_1, p_0) = \lim_{i \to \infty}p_i.
\]
It follows from property (2) that $\Phi$ is injective. Since $X$ is compact and $\mathbb{S}^n$ is Hausdorff, $\Phi$ is a homeomorphism onto its image.

Now we check that $\bigcap_kX_k$ is an $(n-1)$-dimensional Sierpiński curve. It suffices to show that $\bigcup_kD_k$ is dense; the other conditions are already satisfied by construction. If there is an open ball $B \subset \overline{\mathbb{S}^n \setminus \bigcup_kD_k}$ of diameter $\epsilon$, we can fix some $k \geq 1$ such that $\epsilon > \frac{1}{2^{k-1}}$. Consider the maps
\[
\{f_{k+m}f_{k+m-1} \cdots f_k \mid m=1,2,\cdots\}.
\]
By property (1), this sequence converges pointwise to a map $f \colon \bigcap_kX_k \to X_k$ such that $d(f(x),x)<\frac{1}{2^{k-1}}$ for every $x \in \bigcap_kX_k$. Thus $f(B)$ contains a nonempty open set in $X_k$, and hence intersects the dense set $F_k^{-1}(Z)$. However, this contradicts with the assumption that $B \cap \overline{\bigcap_kD_k} = \varnothing$. Thus $\bigcap_kX_k$ is an $(n-1)$-dimensional Sierpiński curve.

Now we prove that $\Phi$ is a homeomorphism. Note that $\bigcup_k \partial D_k \subset \Phi (X)$ since
\[
(\cdots,p_k,p_k,f_k(p_k), f_kf_{k-1}(p_k), \cdots, f_kf_{k-1}\cdots f_1(p_k)) \in X
\]
for $p_k \in \partial D_k$. Since $X$ is compact, $\Phi(X)$ contains a dense subset $\bigcup_k \partial D_k$ in $\bigcap_kX_k$, and $\Phi$ is a homeomorphism to its image, we can conclude that $\Phi$ is a homeomorphism. In the following discussion, we identify $X$ with $\bigcap_kX_k$ and construct the group action on $X$.

Define the projection $F \colon X \to X_0$ by $F(\cdots,p_2,p_1,p_0) = p_0$. Note that the restriction of $F$ on $X \setminus F^{-1}(Z)$ gives a homeomorphism
\[
F \colon X\setminus F^{-1}(Z) \to X_0 \setminus Z
\]
since $X\setminus F^{-1}(Z)$ and $X_0 \setminus Z$ are both compact metric spaces.

Now we lift an action $\rho$ of $\Gamma$ on $X_0$ to an action $\rho'$ on $X$ such that their restriction on $X\setminus F^{-1}(Z)$ agree. This $\rho'$ is given by the identification of $F^{-1}(z_i)$ with the oriented projective space of $T_{z_i}(X_0)$. By our descrition of $X_i$ above, it is by injection. It remains to show that this action is by homeomorphisms. Let $\gamma \in \Gamma$ and $x_i \to x_\infty$ in $X$. If $x_\infty \notin F^{-1}(Z)$, then the continuity follows- direcetly from construction. If $x_\infty \in F^{-1}(z_j)$ for some $j$, we consider the projection of $\{x_i\}$ in $X_j$. Let $v$ be the tangent direction to $X_0$ at $z_j$ representing $x_\infty$. Then we have $F_j(x_i) \to z_j$ and $\frac{F_j(x_i)-z_j}{\|F_j(x_i)-z_j\|} \to v$. Since $\rho(\gamma)$ is a diffeomorphism, $\rho'(\gamma)$ is continuous at $x_\infty$. Hence $\rho'$ is the desired action of $\Gamma$ on $X$ by homeomorphisms.

Recall from the discussion before this proof that collapsing each boundary component of $X \approx \bigcap_k X_k$ to a single point gives a sphere $\bar X$. Then by the discussions in the two paragraphs above, $F$ and the collapsing map together give the desired map $h \colon \mathbb{S}^n \to \mathbb{S}^n$. The action on $X$ constructed above naturally extends to the desired action on the $n$-sphere. By the description of $F$ above, the first three conditions of $h$ in Theorem \ref{Blow_up} are already satisfied. The last condition follows from property (1) of the $\{f_k\}$.
\end{proof}

Let $\mathbb{P}^+(V)$ denote the oriented projective space of a finite-dimensional vector space $V$. In order to apply the above theorem to the natural action of $\mathrm{Sp}(2g,\mathbb{Z})$ on $\mathbb{P}^+(\mathbb{R}^{2g}) \approx \mathbb{S}^{2g-1}$ later, we prove in the following that this action satisfies the assumptions in Theorem \ref{Blow_up}.

\begin{lemma}\label{dense_orbit}
Let $g \geq 1$. Then for the natural action of $\mathrm{Sp}(2g,\mathbb{Z})$ on the oriented projective space $\mathbb{P}^+(\mathbb{R}^{2g}) \approx \mathbb{S}^{2g-1}$, the set of points in $\mathbb{P}^+(\mathbb{R}^{2g})$ with trivial stabilizer and dense orbit is a dense $G_\delta$ set.
\end{lemma}

\begin{proof}
Let $E$ be the set described above. Let $G=\mathrm{Sp}(2g,\mathbb{Z})$ and $X = \mathbb{P}^+(\mathbb{R}^{2g})$. Define
\[
E_1 = \{x \in X \mid \mathrm{Stab}_G(x) = \{1\} \}, \quad E_2 = \{x \in X \mid \overline{G \cdot v} = X \}.
\]
Then $E = E_1 \cap E_2$. It suffices to prove that $E_1$ and $E_2$ are both dense $G_\delta$ sets in X.

For every $g \in G$, let its fixed point set be
\[
\mathrm{Fix}(g) = \{x \in X \mid g \cdot x = x\}.
\]
By continuity of the action, $\mathrm{Fix}(g)$ is closed. Note that every $x \in \mathrm{Fix}(g)$ represents an eigenspace corresponding to a positive eigenvalue. If $\mathrm{int}(\mathrm{Fix}(g))$ is nonempty, $g$ must have a $2g$-dimensional eigenspace corresponding to a positive eigenvalue. However, the only possible positive eigenvalue for $g \in \mathrm{Sp}(2g,\mathbb{Z})$ is $\lambda = 1$. Hence such $g$ must be trivial. In other words, $X \setminus \mathrm{Fix}(g)$ is dense for every $g \neq 1$. Since
\[
E_1 = \bigcap_{g \in G \setminus \{1\}} (X \setminus \mathrm{Fix}(g))
\]
is an intersection of countably many open dense subsets of $X$, $E_1$ is a dense $G_\delta$ set by Baire Category Theorem \cite{Rudin}.

To prove $E_2$ is also a dense $G_\delta$ set, we first prove that the action of $G$ on $X$ is topological transitive, that is, for every nonempty open sets $U,V \subset X$, $\exists g \in G$ such that $g(U) \cap V \neq \varnothing$. Since the image of the primitive vectors in $\mathbb{Z}^{2g}$ are dense in $X$, it suffices to prove that $G$ acts transitively on the primitive vectors, that is, for every $v,w \in \mathbb{Z}^{2g}$, $\exists g \in G$ such that $g \cdot v = w$. Recall that a nonzero element of $H_1(S_g;\mathbb{Z}) \cong \mathbb{Z}^{2g}$ is represented
by an oriented simple closed curve if and only if it is primitive; see \cite[Proposition 6.2]{Primer}. Then by the change of coordinates principle \cite[Section 1.3]{Primer}, there exists an orientation-preserving homeomorphism $f$ which takes the simple closed curve representing $v$ to that representing $w$. Hence the image of $f$ under the symplectic representation takes $v$ to $w$.

Now choose a countable topological basis $\{U_i\}_{i=1}^\infty$ of $X$. Then $x \in E_2$ if and only if $\forall k \geq 1, \exists g \in G, g \cdot x \in U_k$. Let $W_k = \bigcup_{g \in G} g^{-1} U_k$. Then $W_k$ is open and $E_2 = \bigcap_{k=1}^\infty W_k$. Since $G$ acts on $X$ topological transitively, $W_k$ is dense in $X$. Hence $E_2$ is a dense $G_\delta$ set by Baire Category Theorem \cite{Rudin}.
\end{proof}

\section{Faithful Action of a Left-Orderable Group on a closed ball}\label{faithful_action_on_closed_ball}

In this section, we explain the connections between the algebraic properties and the dynamical properties of left-orderable groups. The following result is classic.

\begin{lemma}[Boyer and Clay, Proposition 5.1, \cite{left-orderable}]\label{act_on_R}
A nontrivial countable group $G$ is left-orderable if and only if there exists a faithful representation $\rho \colon G \to \mathrm{Homeo}^+(\mathbb{R})$.
\end{lemma}

\begin{proof}
If there exists a faithful representation $\rho \colon G \to \mathrm{Homeo}^+(\mathbb{R})$, we can choose a dense set $\{q_i\}_{i=1}^{\infty} \subset \mathbb{R}$. Since $\rho$ is faithful and $\{q_i\}_{i=1}^{\infty}$ is dense, for every $x,y \in G$, there exist $i \in \mathbb{Z}_+$ such that $\rho(x)(q_i) \neq \rho(y)(q_i)$. We define a total order $<$ on $G$ as the following: $x<y$ if there exists $k \in \mathbb{Z}_+$ such that $\rho(x)(q_i)=\rho(y)(q_i)$ for $1 \leq i \leq k-1$, and $\rho(x)(q_k) < \rho(y)(q_k)$. One can check that it is a well-defined total order. Note that for every $g \in G$, $\rho(g) \in \mathrm{Homeo}^+(\mathbb{R})$ preserves orientation of $\mathbb{R}$. If $x<y$ and $k \in \mathbb{Z}^+$ satisfies the conditions above, then for every $g \in G$,
\[
\rho(gx)(q_i) = \rho(g)(\rho(x)(q_i)) = \rho(g)(\rho(y)(q_i)) = \rho(gy)(q_i),
\]
for every $1 \leq i \leq k-1$, and
\[
\rho(gx)(q_k) = \rho(g)(\rho(x)(q_k)) < \rho(g)(\rho(y)(q_k)) = \rho(gy)(q_k).
\]
Hence $gx<gy$, and $G$ is left-orderable.

Conversely, if $G$ is left-orderable with total order $<$, we write $G=\{g_0, g_1, g_2 \cdots\}$, where $g_0 = 1$. We define a map $t \colon G \to \mathbb{R}$ by induction as the following. Let $t(g_0)=0$. If $t(g_i)$ is already defined for $0 \leq g \leq k$, we define
\[
t(g_{k+1}) = \begin{cases}
\max\{t(g_0), \dots, t(g_k)\} + 1 & \text{if } g_{k+1} > \max\{g_0, \dots, g_k\} \\
\min\{t(g_0), \dots, t(g_k)\} - 1 & \text{if } g_{k+1} < \min\{g_0, \dots, g_k\} \\
\frac{t(g_i)+t(g_j)}{2} & \text{if } g_i < g_{k+1} < g_j \text{ and } \nexists \, 0 \leq n \leq k \text{ s.t. } g_i < g_{n} < g_j
\end{cases}
\]
Let $G$ act on $t(G)$ by $g(t(h)) = t(gh)$. This action extends uniquely to the closure $\overline{t(G)}$. Since $\mathbb{R} \setminus \overline{t(G)}$ is a disjoint union of open intervals, we can extend the action affinely over $\mathbb{R} \setminus \overline{t(G)}$, which gives a faithful representation $\rho \colon G \to \mathrm{Homeo}^+(\mathbb{R})$ as desired.
\end{proof}

To make use of this dynamical property in our construction, we modify the action constructed in the proof of the above lemma to construct a faithful action on the closed $m$-ball.

\begin{theorem}\label{act_faithfully_on_closed_disks}
Let $K$ be a nontrivial countable left-orderable group. Let $m \geq 1$. Then there exists an injective homomorphism
\[
\theta \colon K \to \mathrm{Homeo}^+(D^m,\partial D^m),
\]
where $D^m$ is the closed $m$-ball.
\end{theorem}

\begin{proof}
By Lemma \ref{act_on_R}, there exists an injective homomorphism $\rho \colon K \to \mathrm{Homeo}^+(\mathbb{R})$. Take an orientation-preserving homeomorphism $\phi \colon \mathbb{R} \to (0,1)$. For every $k \in K$, define $h_k \colon (0,1) \to (0,1)$ by
\[
h_k(x) = \begin{cases}
\phi \rho(k) \phi^{-1}(x) & \text{if } x \in (0,1) \\
x & \text{if } x \notin (0,1)
\end{cases}
\]
Then define $H_k \colon \mathbb{R}^m \to \mathbb{R}^m$ by
\[
H_k(x_1, x_2, \cdots, x_m) = (h_k(x_1), x_2, \cdots, x_n).
\]
Then $H_k$ is a well-defined orientation-preserving homeomorphism, and the map $k \mapsto H_k$ is an injective homomorphism. Let $j \colon \mathbb{R}^m \to \mathrm{int}(D^m)$ be the homomorphism $j(x) = \frac{x}{1+\|x\|}$, where $x=(x_1, x_2, \cdots, x_n)$. Define $\theta \colon K \to \mathrm{Homeo}^+(D^m)$ by
\[
\theta(k)(x)=\begin{cases}
jH_kj^{-1}(x) &\text{if }x \in \mathrm{int}(D^m)\\
x &\text{if }x \in \partial D^m
\end{cases}
\]
Suppose $j(y_i) \to u \in \partial D^m$. Then $\|y_i\| \to \infty$ and $\frac{y_i}{\|y_i\|}$. Note that $\|H_k(x)-x\| \leq 1$ for every $x \in \mathbb{R}^m$. Then
\[
\lim_{i \to \infty}\theta(k)(j(y_i)) = \lim_{i \to \infty} \frac{H_k(y_i)}{1+\|H_k(y_i)\|} = \lim_{i \to \infty} \frac{y_i}{\|y_i\|} = u.
\]
Hence $\theta(k)$ is continuous. Similarly, we can prove that $\theta(k)^{-1}$ is also continuous. Hence $\theta(k) \in \mathrm{Homeo}^+(D^m,\partial D^m)$ and $\theta$ is an injective homomorphism, as desired.
\end{proof}

Note that the Torelli group $\mathcal{I}(S_g)$ is countable since $\mathrm{Mod}(S_g)$ is finitely generated and hence is countable. Then we can apply Theorem \ref{act_faithfully_on_closed_disks} to obtain a faithful action of $\mathcal{I}(S_g)$ on a closed $m$-ball.

\section{Faithful Action of \texorpdfstring{$\mathrm{Mod}(S_g)$}{Mod(Sg)} on \texorpdfstring{$\mathbb{S}^{2g-1}$}{S2g-1}}\label{proof_of_main_thm}

Now we use the short exact sequence given by the symplectic representation of $\mathrm{Mod}(S_g)$ to extend the action of $\mathrm{Sp}(2g,\mathbb{Z})$ on the blown-up space given by Theorem \ref{Blow_up} to the desired faithful action of $\mathrm{Mod}(S_g)$ on $\mathbb{S}^{2g-1}$. More generally, we have the following:

\begin{theorem}\label{extend}
Let
\[
1 \longrightarrow K \longrightarrow G \stackrel{\pi}{\longrightarrow} Q \longrightarrow 1
\]
be a short exact sequence of groups with quotient homomorphism $\pi \colon G \to Q$. Let $m \geq 1$ and $X$, $X'$ be connected closed oriented $m$-manifolds. Let
\[
\rho \colon Q \to \mathrm{Homeo}^+(X), \quad \rho' \colon Q \to \mathrm{Homeo}^+(X')
\]
be group actions. Assume that there exists $x_0 \in X$ with trivial $Q$-stabilizer such that its $Q$-orbit $Z$ is dense in $X$. Assume that there is a surjective continuous map $h \colon X' \to X$ such that:
\begin{enumerate}
  \item $h\rho'(q) = \rho(q)h$ for every $q \in Q$;
  \item $B_r = h^{-1}(\rho(r)x_0)$ is a topologically embedded closed $m$-ball for every $r \in Q$;
  \item $h^{-1}(x)$ is a single point for every $x \in X \setminus Z$;
  \item $\{B_r : r \in Q\}$ is a null family for some compatible metric on $X'$.
\end{enumerate}
Suppose that there is an injective homomorphism
\[
\theta \colon K \to \mathrm{Homeo}^+(D^m,\partial D^m).
\]
Then there exists an injective homomorphism
\[
f \colon G \to \mathrm{Homeo}^+(X'),
\]
such that $hf(g)=\rho(\pi(g))h$ for every $g \in G$.
\end{theorem}

\begin{proof}
Fix a section $s \colon Q \to G$ such that $s(1) = 1$ and $\pi(s(r))=r$ for every $r \in Q$. The equivariance of $h$ implies that $\rho'(q)(B_r) = B_{qr}$ for every $q,r \in Q$. Choose an orientation-preserving parameterization $t_1 \colon D^m \to B_1$ of the embedded ball $B_1$. Define $t_r \colon D^m \to B_r$ by $t_r = \rho'(r)t_1$. Then $t_r$ is an orientation-preserving parameterization for every $r \in Q$. For every $g \in G$ and $r \in Q$, define
\[
k(g,r) = s(\pi(g)r)^{-1}gs(r).
\]
Since
\[
\pi(k(g,r)) = (\pi(g)r)^{-1} \pi r = 1,
\]
we have $k(g,r) \in K$. Now we define $f(g)$ for every $g \in G$ as follows. For every $r \in Q$, let
\[
f(g)\big|_{B_r}=t_{\pi(g)r}\theta(k(g,r))t_r^{-1}.
\]
Then let
\[
f(g)\big|_{X' \setminus \bigcup_r \operatorname{int}(B_r)} = \rho'(\pi (g)).
\]
These definitions agree on every $\partial B_r$ because $\rho'(\pi (g)) t_r = t_{\pi(g)r}$ and $\theta(k(g,r))$ fixes $\partial D^m$ pointwise.

Now we prove that $f(g)$ is continuous. Let $y_n \to y \in X'$. If $y \in \mathrm{int}(B_r)$ for some $r \in Q$, then eventually $y_n \in \mathrm{int}(B_r)$ and continuity follows from the continuity of $t_{\pi(g)r}$, $\theta(k(g,r))$ and $t_r^{-1}$. Now suppose $y\in X' \setminus \bigcup_r \operatorname{int}(B_r)$. By the definition of $f(g)$, it suffices to consider the subsequence of the terms $y_n$ that lie in some $\mathrm{int}(B_r)$. Replace $\{y_n\}$ with this subsequence and let $y_n \in \mathrm{int}(B_{r_n})$. If there exists a subsequence $\{y_{n_k}\}$ such that $r_{n_k}=r$ for some $r \in Q$, then $y \in \partial B_r$ and the continuity follows from the continuity in $\mathrm{int}(B_r)$ and the definitions of $f(g)$ agree on $\partial B_r$. Otherwise, we can replace $\{y_n\}$ with a subsequence such that $\{r_n\}$ are pairwise distinct. Note that $f(g)y_n, \rho(\pi(g))y_n \in B_{\pi(g)r_n}$. Since $\{B_r : r \in Q\}$ is a null family, $\mathrm{diam}(B_{\pi(g)r_n}) \to 0$. Then $\rho(\pi(g))y_n \to \rho(\pi(g))y$ gives $f(g)y_n \to \rho(\pi(g))y = f(g)y$. Hence $f(g)$ is continuous.

For every $g_1, g_2 \in G$,
\[
k(g_1g_2,r) = s(\pi(g_1) \pi(g_2)r)^{-1} g_1 s(\pi(g_2)r) s(\pi(g_2)r)^{-1} g_2r = k(g_1,\pi(g_2)r) k(g_2,r).
\]
Then $f(g_1) f(g_2) =f(g_1g_2)$ on $B_r$. The same equality also holds on $X' \setminus \bigcup_r \operatorname{int}(B_r)$ due to definition. Note that $f(1)$ is the identity. Then $f$ is a homomorphism. Hence $f(g)$ and $f(g)^{-1}$ are both continuous and $f(g)$ is a homeomorphism.

Since $t_{\pi(g)r}\theta(k(g,r))t_r^{-1}$ preserves local orientation for every $g \in G$ and every $r \in Q$, $f(g)$ also preserves orientation. Hence $f(g) \in \mathrm{Homeo}^+(X')$.

Since $x_0$ has a trivial stabilizer and $\theta$ is injective, $f$ is also injective.

Finally, if $y \in B_r$, then
\[
h(f(g)y) = \rho(\pi(g)r)x_0 = \rho(\pi(g))h(y).
\]
If $y \in X' \setminus \bigcup_r \operatorname{int}(B_r)$, then
\[
hf(g) = h \rho'(\pi(g)) = \rho(\pi(g))h.
\]
Hence $hf(g)=\rho(\pi(g))h$ for every $g \in G$.
\end{proof}

With these results established, the proof of Theorem \ref{main} is now straightforward.

\begin{proof}[Proof of Theorem \ref{main}]
Recall that we have the short exact sequence
\[
1 \longrightarrow \mathcal{I}(S_g) \longrightarrow \mathrm{Mod}(S_g) \stackrel{\Psi}{\longrightarrow} \mathrm{Sp}(2g,\mathbb{Z}) \longrightarrow 1.
\]
We apply Theorem \ref{extend} to this short exact sequence. Let $m=2g-1$, $X = \mathbb{P}^+(\mathbb{R}^{2g})$, and $\rho$ be the natural action of $\mathrm{Sp}(2g,\mathbb{Z})$ on $X$. By Lemma \ref{dense_orbit}, there exists $x_0 \in X = \mathbb{P}^+(\mathbb{R}^{2g})$ such that $x_0$ has a trivial stabilizer and a dense orbit $Z$ in $X$ with respect to the action $\rho$. Then by Theorem \ref{Blow_up}, there exists an action $\rho'$ of $\mathrm{Sp}(2g,\mathbb{Z})$ on $X' \approx \mathbb{S}^{2g-1}$ and $h \colon X' \to X$ that satisfies all the corresponding assumptions in Theorem \ref{extend}. By Theorem \ref{act_faithfully_on_closed_disks}, there exists an injective homomorphism $\theta \colon \mathcal{I}(S_g) \to \mathrm{Homeo}^+(D^m,\partial D^m)$. Hence Theorem \ref{extend} gives an injective homomorphism $f \colon \mathrm{Mod}(S_g) \to \mathrm{Homeo}^+(\mathbb{S}^{2g-1})$, as desired.
\end{proof}

\section*{Disclosure of AI Use}

In the preparation of this work, ChatGPT and Gemini were used to improve language clarity and correct grammatical mistakes.

\section*{Acknowledgments}

I would like to thank Professor Yi Liu for introducing this topic to me and giving me valuable advice on this paper. I also thank J. Bowden and K. Mann for taking the time to read an earlier version of this paper and for their helpful feedback.

\bibliography{references}{}
\bibliographystyle{alpha}

\end{document}